\documentclass[ECP,preprint]{ejpecp} 

\SHORTTITLE{A data-dependent DKW inequality for regenerative Markov chains}

\TITLE{A data-dependent DKW inequality for regenerative Markov chains}

\AUTHORS{%
  Daniel~C.~Jerison\footnote{University of San Francisco.
    \EMAIL{dcjerison@usfca.edu}}
  }


\AMSSUBJ{60J05} 

\SUBMITTED{July 11, 2026} 
\ACCEPTED{N/A} 

\VOLUME{0}
\YEAR{20XX}
\PAPERNUM{0}
\DOI{10.1214/YY-TN}

\ABSTRACT{We prove a version of the Dvoretzky--Kiefer--Wolfowitz inequality for Markov chains with a regenerative structure. Suppose we have a regenerative Markov chain with stationary distribution $\pi$. Given a functional $\theta$ on the state space and a confidence level $1-\delta$, our result provides a uniform $1-\delta$ confidence band for the CDF of $\theta$ under $\pi$ based on the empirical CDF. By inversion, we get a $1-\delta$ confidence band for the quantile function of $\theta$ under $\pi$.

Our bounds are fully explicit and nearly optimal. In addition, they are data-dependent in the following sense: in the formula for the width of the confidence band, the leading term can be computed directly from the sample path without any \emph{a priori} information about the convergence rate of the chain. A convergence bound is required, but it contributes to the width of the confidence band only through a lower-order term. For this reason, our result is attractive for Markov chains whose convergence rate is much quicker in practice than what can be proved in theory.

Data-dependent bounds of this type are called empirical concentration inequalities in the literature. Thus, our result is an empirical concentration inequality for the empirical CDF of $\theta$ given the sample path.}

\newcommand{\R}{\mathbf{R}}
\newcommand{\Prob}{\mathbf{P}}
\newcommand{\eee}{\mathbf{E}}
\newcommand{\X}{\mathcal{X}}
\newcommand{\1}{\mathbf{1}}
\newcommand{\e}{\varepsilon}
\DeclareMathOperator{\Var}{\mathrm{Var}}
\newcommand{\W}{\mathcal{W}}

\begin{document}



\section{Introduction} \label{sec:introduction}

This paper is concerned with explicit, nonasymptotic bounds for Markov chain Monte Carlo estimation. Suppose that $(X_t)_{t \geq 0}$ is a discrete time Markov chain on a state space $\X$ with stationary distribution $\pi$. Let $\theta: \X \to \R$ be an observable quantity of interest. The empirical distribution of $\theta$ given the sample path $(X_0,\ldots,X_{n-1})$ is
\[
\hat{\pi}_n(\theta \in S) = \frac{1}{n} \sum_{t=0}^{n-1} \1\{\theta(X_t) \in S\}, \qquad \text{for measurable $S \subseteq \R$}.
\]
A DKW inequality for the Markov chain $(X_t)$ is a quantitative upper bound on
\[
\sup_{x \in \R} \left| \hat{\pi}_n(\theta \leq x) - \pi(\theta \leq x) \right|
\]
that holds with high probability. It is named after the classical inequality of Dvoretzky--Kiefer--Wolfowitz \cite{DKW56}, which provides the same type of bound in the case of iid data.

Various authors \cite{DMR95,GR14,KW14,P15,BC25} have proved DKW-type inequalities for Markov chains (or, more generally, for dependent sequences satisfying strong mixing conditions) under different assumptions. In any such result, it is necessary to have some control over the mixing properties of the chain. This is a major issue:
\begin{itemize}
\item Many chains used in the practice of MCMC have no theoretical convergence guarantees, or have only qualitative convergence guarantees (e.g., the chain is known to be geometrically ergodic but without any explicit convergence rate).
\item Some chains have theoretical convergence guarantees with explicit rates, but the numerical convergence bounds are very conservative compared with the observed behavior of the chain.
\end{itemize}

When dealing with a Markov chain in the second category, we might hope to get tighter estimation error bounds using information from the sample path. Our main result, Theorem \ref{thm:main}, accomplishes this goal for chains with a regenerative structure.

Let $\nu$ be a probability measure on the state space $\X$. A \emph{$\nu$-regeneration time} for the chain $(X_t)$ is a random time $T$ with the property that $X_T \sim \nu$, and moreover, the law of $X_T$ is independent of the value of $T$ and of the history $(X_t)_{t<T}$. In other words, at time $T$ the chain regenerates in a state drawn independently from $\nu$.

We will assume that the chain $(X_t)$ is started from $X_0 \sim \nu$ and has a sequence $0 = T_0,T_1,T_2,\ldots$ of $\nu$-regeneration times. These times partition the sample path $(X_t)_{t \geq 0}$ into tours $(X_t)_{T_{k-1} \leq t < T_k}$ with the property that each tour starts at a regeneration and ends just before the next regeneration. Theorem \ref{thm:main} assumes that the tours are iid and that the tail of the tour length distribution decays exponentially.

\begin{theorem} \label{thm:main}
Let $(X_t)_{t \geq 0}$ be a discrete time Markov chain on the state space $\X$ with stationary distribution $\pi$, and let $\theta: \X \to \R$ be any observable quantity associated with the chain. For some fixed probability measure $\nu$ on $\X$, assume that $(X_t)$ is equipped with a sequence $0 = T_0,T_1,T_2,\ldots$ of $\nu$-regeneration times such that the tours $(X_t)_{T_{k-1} \leq t < T_k}$ are iid. (In particular, this means that $X_0 \sim \nu$.) Write $T = T_1$, and assume that
\begin{equation} \label{eq:exp-bound}
\Prob(T>t) \leq Be^{-\gamma t} \qquad \text{for all } t \geq 0
\end{equation}
for some fixed positive constants $B,\gamma$.

Fix an integer $K \geq 1$, and let $n = T_K$ be the $K$-th regeneration time. The first $K$ tour lengths are $W_k = T_k - T_{k-1}$ for $1 \leq k \leq K$. Let $M = \max(W_1,\ldots,W_K)$ be the maximum of these tour lengths. Also, for each integer $j \geq 0$, let $m_j = \#\{1 \leq k \leq K : W_k > j\}$ be the number of tour lengths that exceed $j$.

Let
\[
\hat{\pi}_n(\theta \leq x) = \frac{1}{n} \sum_{t=0}^{n-1} \1\{\theta(X_t) \leq x\}
\]
be the empirical CDF of $\theta$ given the sample path $(X_0,\ldots,X_{n-1})$. Then, for every $\delta > 0$, with probability at least $1-\delta$ we have
\begin{align}
\sup_{x \in \R} \left| \hat{\pi}_n(\theta \leq x) - \pi(\theta \leq x) \right| \notag \\
&\leq \frac{1}{\sqrt{2n}} \sum_{j=0}^{M-1} \sqrt{ \frac{m_j}{n} \log\left( \frac{6n}{\delta m_j} \right) \wedge \frac{2m_j^2}{n}} + \frac{\sqrt{K}}{n} \cdot \frac{M}{\sqrt{\delta/3}} \label{eq:bound-1} \\
&\qquad + \frac{2}{n \sqrt{\delta/3}} \left( \frac{\log(BK)}{\gamma} \sqrt{\log\left(\frac{3}{\delta} \right)} + \frac{1}{1 - e^{-\gamma/2}} \right). \label{eq:bound-2}
\end{align}
\end{theorem}

\subsection{Size of the upper bound} \label{sec:size-upper-bound}

We now discuss the sizes of the terms \eqref{eq:bound-1} and \eqref{eq:bound-2}. Regarding the leading term \eqref{eq:bound-1}, a direct application of Jensen's inequality leads to the following bound.

\begin{lemma} \label{lemma:Jensen}
In the context of Theorem \ref{thm:main},
\[
\sum_{j=0}^{M-1} \sqrt{ \frac{m_j}{n} \log\left( \frac{6n}{\delta m_j} \right) \wedge \frac{2m_j^2}{n}} \leq \sqrt{M \log\left( \frac{6M}{\delta} \right)}.
\]
\end{lemma}

From Lemma \ref{lemma:Jensen} it follows that
\[
\eqref{eq:bound-1} \leq \frac{M}{\sqrt{n}} \left[ \sqrt{\frac{1}{2M} \log\left( \frac{6M}{\delta} \right)} + \frac{\sqrt{K}}{\sqrt{n}} \cdot \frac{1}{\sqrt{\delta/3}} \right].
\]
Since $M$ is the maximum of $K$ iid tour lengths satisfying the exponential tail bound \eqref{eq:exp-bound}, it has size $\log(K) \leq \log(n)$. Thus, \eqref{eq:bound-1} has size $\log(n)/\sqrt{n}$ multiplied by a term that does not grow with $n$. Crucially, \eqref{eq:bound-1} is fully data-dependent: its size is determined by the actual frequency of regenerations observed in the sample path $(X_0,\ldots,X_{n-1})$ and not by the constants $B,\gamma$ in \eqref{eq:exp-bound}, which may be quite conservative. The dependence on $n$ is optimal except possibly for the logarithmic factor.\footnote{To the author's knowledge, all similar results either also have a log factor or impose an extra condition, such as uniform ergodicity, that allows the log factor to be removed. The bound in \cite{BC25} has a factor of $\sqrt{\log(n)}$. It is unknown whether the $\log(n)$ in \eqref{eq:bound-1} could be improved to $\sqrt{\log(n)}$ while preserving the data-dependent nature of the bound.}

The lower-order term \eqref{eq:bound-2} has the form
\begin{equation} \label{eq:big-rhs}
\frac{\log(BK)}{n\gamma}
\end{equation}
multiplied by a constant depending on $\delta$ (in the regime where $\log(BK)/\gamma \geq 1$). This has size $\log(n)/n$. We show in Section \ref{sec:sharpness} that the term \eqref{eq:big-rhs} is necessary under the assumptions of Theorem \ref{thm:main}. Thus, \eqref{eq:bound-2} has optimal dependence on $n,K$ and on $B,\gamma$.

The dependence on $\delta$ of \eqref{eq:bound-1} and \eqref{eq:bound-2} is not optimal, due to a step in the proof that uses Markov's inequality. We could instead invoke an exponential concentration inequality, which would replace the $\delta^{-1/2}$ terms in \eqref{eq:bound-1} and \eqref{eq:bound-2} with $\log(1/\delta)$ terms but worsen the dependence on $B,\gamma$. See Lemmas \ref{lemma:Wint-bound} and \ref{lemma:Markov} for details.

\subsection{Small sets} \label{sec:small-sets}

A standard way to construct the regeneration times required for Theorem \ref{thm:main} is to identify a \emph{small set}. This is a subset $A \subseteq \X$ together with a constant $\e>0$ and a probability measure $\nu$ on $\X$ such that
\[
P(x,\cdot) \geq \e \cdot \nu(\cdot) \qquad \text{for all } x \in A,
\]
where $P(x,\cdot)$ is the Markov transition kernel associated with the chain $(X_t)$. Given a small set $A$ which is visited infinitely often almost surely, the ``Nummelin splitting'' technique defines $(X_t)$ and the regeneration times $T_k$ on the same probability space. Every time the chain visits a state $x \in A$, flip an independent $(\e,1-\e)$ coin. If it shows $\e$, jump to $\nu$ and regenerate. Otherwise, sample the next state from the remainder measure $\frac{1}{1-\e}[P(x,\cdot) - \e\nu(\cdot)]$. For details, see \cite{MTY95,N84}.

To apply Theorem \ref{thm:main}, we must be able to identify the regeneration times associated with a given sample path $(X_0,\ldots,X_{n-1})$. Finding an appropriate set $A$ and measure $\nu$ may be quite difficult. But, once $A$ and $\nu$ have been chosen, identifying the regeneration times is often straightforward \cite{MTY95,AG07}. (In particular, it is not necessary to compute the normalizing constant for $\nu$.) The requirement that $X_0 \sim \nu$ is easily satisfied: simply run the chain until the first regeneration and discard the previous samples.

We note that the regeneration times required by Theorem \ref{thm:main} may be constructed using a small set or by any other procedure. The only requirements are that the tours must be iid, the tail of the tour length distribution must satisfy the exponential bound \eqref{eq:exp-bound}, and it must be possible to identify the regeneration times in simulation.

There is a vast literature on regenerative Markov chains. Much of it focuses on using regenerations to prove convergence bounds. Briefly, three ingredients are needed: (1) the regenerative structure, which may be provided by a small set; (2) some control over the frequency of regenerations, which may be provided by a so-called \emph{drift function} or \emph{Lyapunov function}; and (3) an aperiodicity condition (such as reversibility with nonnegative eigenvalues). See \cite{MT12,J16} for more information. Given a drift function, the constants $B,\gamma$ in \eqref{eq:exp-bound} can be computed using \cite[Theorem 4.9]{J16}.

In this paper, we rely on ingredients (1) and (2) only. These are not enough to bound the convergence rate of the chain. In fact, the chain could enjoy frequent regenerations even if it is periodic. This does not pose any problem for Theorem \ref{thm:main}, because the periodicity is washed out by the time average. Accordingly, we will need no extra hypotheses such as reversibility or aperiodicity to establish Theorem \ref{thm:main}.

\subsection{Empirical concentration inequalities} \label{sec:empirical}

An empirical concentration inequality (or data-dependent concentration inequality) is one that obtains sharper bounds by using properties of the random sample. To illustrate the general form, we describe a result of Maurer and Pontil \cite{MP09}.

Let $Z_1,\ldots,Z_n$ be iid random variables supported on $[0,1]$ with mean $\eee(Z)$ and variance $\Var(Z)$. Write $\overline{Z}_n$ and $\hat{V}_n(Z)$ for the sample mean and sample variance of $Z_1,\ldots,Z_n$. For all $\delta>0$, with probability at least $1-\delta$,
\begin{equation} \label{eq:empirical-Bennett}
\overline{Z}_n - \eee(Z) \leq \sqrt{\frac{2 \hat{V}_n(Z) \log(2/\delta)}{n}} + \frac{7 \log(2/\delta)}{3(n-1)}.
\end{equation}

The bound \eqref{eq:empirical-Bennett} differs from the classical Bennett's inequality in that the sample variance $\hat{V}_n(Z)$ replaces $\Var(Z)$ in the leading term. It is instructive to compare the form of \eqref{eq:empirical-Bennett} with our Theorem \ref{thm:main}. In both cases, the leading term is purely data-dependent. This makes the inequality tighter when the random sample is ``better-behaved'' than we could expect from any \emph{a priori} information. In \eqref{eq:empirical-Bennett}, the improvement happens when the sample variance is much less than $1/4$; in Theorem \ref{thm:main}, the improvement happens when the sample path regenerates much more frequently than is guaranteed by \eqref{eq:exp-bound}.

Many authors have developed empirical concentration inequalities in the vein of \eqref{eq:empirical-Bennett} in various contexts. The recent work \cite{MTR25} provides a substantial quantitative improvement on the techniques of \cite{MP09}. We highlight two papers that extend the theory of empirical concentration inequalities in the direction of Markov chains and dependent sample data. Wintenberger \cite{W17} considers a geometrically ergodic Markov chain with a small set that provides regenerations and a drift function $V$ that controls their frequency. He proves an empirical Bernstein inequality that can be used, for example, to bound the convergence of $\frac{1}{n} \sum_{t=0}^{n-1} V(X_t)$ to $\eee_\pi(V)$. Mirzaei et al \cite{MMKP25} extend the inequality \eqref{eq:empirical-Bennett} of \cite{MP09} to the setting of dependent sample data that satisfies a decay condition of the $\beta$-mixing coefficients. Since the $\beta$-mixing coefficients of a geometrically ergodic Markov chain decay exponentially \cite{R17}, this leads to an empirical concentration inequality for the convergence of $\frac{1}{n} \sum_{t=0}^{n-1} f(X_t)$ to $\eee_\pi(f)$ whenever $(X_t)$ is geometrically ergodic and $f$ is bounded.

For our purposes, the word ``empirical'' has two different senses: as a synonym for ``data-dependent'' in empirical concentration inequalities, and as a descriptor for the empirical CDF in the context of DKW-type bounds. One could characterize Theorem \ref{thm:main} as an empirical concentration inequality for the empirical CDF associated with a functional of a regenerative Markov chain. Aside from the linguistic curiosity, this perspective shows what is new about our work. DKW bounds and empirical concentration inequalities have already been established for Markov chains. Our result is the first to combine both of these in a data-dependent DKW inequality.

\subsection{Sharpness} \label{sec:sharpness}

We now explain why the term \eqref{eq:big-rhs} must appear in the right side of Theorem \ref{thm:main}.

Suppose that the state space $\X$ is partitioned into two regions $\X = \X_1 \cup \X_2$, with $\pi(\X_1) \approx \pi(\X_2) \approx 1/2$, and such that the chain is very unlikely to move between $\X_1$ and $\X_2$. If we define a small set $A$ and regeneration measure $\nu$ on $\X_1$, then the chain started from $\nu$ will enjoy frequent regenerations with short tour lengths as long as it stays within $\X_1$. Eventually, the chain will jump to $\X_2$ and stay there for a long time. This will show up in the sample path as an extremely long tour length.

Suppose that our sample path has gone through $K$ regenerations, so that we have $K$ independent samples from the tour length distribution. Also suppose that all of the observed tour lengths are reasonably short. From this information, we can say that long tours are unusual: with high probability, they occur with frequency at most $c/K$ for some constant $c$. But this says nothing about the length of a long tour once it occurs. In the bimodal example, the length of a long tour is determined by the probability of jumping from $\X_2$ to $\X_1$, which could be arbitrarily small.

This is why we need the \emph{a priori} bound \eqref{eq:exp-bound} on tour length. Suppose that we have observed $K$ regenerations for a total of $n$ samples from $(X_t)$. It could be that long tours occur with frequency $1/K$. Solving for $t$ in $Be^{-\gamma t} = 1/K$ yields
\[
t = \frac{\log(BK)}{\gamma}.
\]
Thus, the bound \eqref{eq:exp-bound} is consistent with the long tours having approximate length $\log(BK)/\gamma$ when they occur. If we append a long tour of this length after a sequence of $K$ short tours of total length $n$, the empirical CDF $\hat{\pi}_n(\theta \leq x)$ could change by up to
\[
\frac{\log(BK)/\gamma}{n + \log(BK)/\gamma}.
\]
Once $K$ and $n$ are large enough, this term equals \eqref{eq:big-rhs} up to a multiplicative constant.

\subsection{Discussion} \label{sec:discussion}

Our main result, Theorem \ref{thm:main}, is a data-dependent DKW inequality for regenerative Markov chains. Since it provides a $1-\delta$ confidence band for the CDF of the functional $\theta$ under $\pi$, by inverting we obtain a $1-\delta$ confidence band for the quantile function of $\theta$. In other words, Theorem \ref{thm:main} allows us to estimate all quantiles of $\theta$ simultaneously.

Various improvements to Theorem \ref{thm:main} are possible. First, perhaps the dependence on $\delta$ could be improved by replacing Lemma \ref{lemma:Wint-bound} with an exponential concentration inequality that does not worsen the dependence on $B,\gamma$. In addition, the result could be extended. Suppose that instead of estimating the CDF of a single functional $\theta$, we care about a family of functionals $\{\theta_s\}_{s \in \mathcal{S}}$. It is likely that empirical process theory could be used to bound
\[
\sup_{s \in \mathcal{S}} \sup_{x \in \R} \left| \hat{\pi}_n(\theta_s \leq x) - \pi(\theta_s \leq x) \right|
\]
under appropriate conditions on $\mathcal{S}$. In a different direction, we could hope for a bound that is uniform over all possible choices of $n$, as has been achieved in \cite{HR22} for the classical DKW inequality.

The techniques in this paper apply only to regenerative chains. What can be said about Markov chains for which the \emph{a priori} convergence bound comes in some other form, such as a bound on the spectral gap? Theorem \ref{thm:main} leverages the sample path $(X_t)$ into structural information about the Markov chain by counting regenerations. In what other ways does the sample path carry information about the chain?

The ultimate goal of results like Theorem \ref{thm:main} is to provide numerical bounds that are tight enough to be practically useful. The example discussed in Section \ref{sec:sharpness} shows the limitations of the regenerative approach: in the absence of additional conditions, the bounds given by Theorem \ref{thm:main} cannot be substantially improved. What extra conditions could we impose that would lead to tighter bounds, and could those conditions be verified for chains of practical interest?

\section{Proof of Theorem \ref{thm:main}} \label{sec:proof-of-main-theorem}

We now prove Theorem \ref{thm:main}. Proofs of several lemmas are postponed to Section \ref{sec:proofs-of-lemmas}.

For each $i>0$, set $p_i = \Prob(T=i)$. Also, for each $j \geq 0$, set
\[
P_j = \Prob(T>j) = \sum_{i>j} p_i \qquad \text{and} \qquad q_j = \frac{P_j}{\eee(T)}.
\]
We have $\sum_{j=0}^\infty q_j = 1$. For each $j \geq 0$, let
\[
\pi_j(\cdot) = \Prob(X_j \in \cdot \mid T>j)
\]
be the conditional law of $X_t$ given that the chain has taken $j$ steps since the most recent regeneration. It is well-known that $\pi = \sum_{j=0}^\infty q_j \pi_j$. This gives a decomposition of CDFs
\begin{equation} \label{eq:cdf-decomp}
\pi(\theta \leq x) = \sum_{j=0}^\infty q_j \pi_j(\theta \leq x).
\end{equation}

We now turn to the empirical distribution of $\theta$. For each $j \geq 0$, let
\[
S_j = \{T_{k-1} + j : 1 \leq k \leq K, \, T_{k-1} + j < T_k \}
\]
be the set of times that occur exactly $j$ steps after a regeneration and before the next regeneration. We observe that $S_j$ contains a time in the $k$-th tour if and only if the tour length satisfies $W_k > j$. Thus, each $|S_j| = m_j$.

Since the maximum tour length is $M$, we have $m_{M-1} > 0$ and $m_M = 0$. The sets $S_j$ form a partition of $\{0,1,\ldots,n-1\}$, so $m_0 + m_1 + \cdots + m_{M-1} = n$. For each $j$, the samples $\{X_t : t \in S_j\}$ are iid with distribution $\pi_j$, due to independence of the tours.

Denote the empirical CDF of $\theta$ by
\[
\hat{F}(x) = \hat{\pi}_n(\theta \leq x) = \frac{1}{n} \sum_{t=0}^{n-1} \1\{\theta(X_t) \leq x\}
\]
and, for each $0 \leq j \leq M-1$,
\[
\hat{F}_j(x) = \frac{1}{m_j} \sum_{t \in S_j} \1\{\theta(X_t) \leq x\}.
\]
Then, we have the decomposition
\begin{equation} \label{eq:emp-cdf-decomp}
\hat{F}(x) = \sum_{j=0}^{M-1} \frac{m_j}{n} \hat{F}_j(x).
\end{equation}
Notice that the decompositions \eqref{eq:cdf-decomp} and \eqref{eq:emp-cdf-decomp} have different weights.

We use \eqref{eq:cdf-decomp} and \eqref{eq:emp-cdf-decomp} to obtain the following bound.

\begin{lemma} \label{lemma:decomp-sum}
\begin{equation} \label{eq:decomp-sum}
\sup_{x \in \R} \left| \hat{F}(x) - \pi(\theta \leq x) \right| \leq \sum_{j=0}^{M-1} \frac{m_j}{n} \sup_{x \in \R} \left| \hat{F}_j(x) - \pi_j(\theta \leq x) \right| + \sum_{j=0}^\infty \left| \frac{m_j}{n} - q_j \right|.
\end{equation}
\end{lemma}

The first sum on the right side of \eqref{eq:decomp-sum} represents the errors in learning the distribution of $\theta$ under each $\pi_j$. Since each $\hat{F}_j(x)$ is computed using iid samples from $\pi_j$, we can use the classical DKW inequality to prove the bound below:

\begin{lemma} \label{lemma:dkw-result} With probability at least $1 - \delta/3$,
\[
\sum_{j=0}^{M-1} \frac{m_j}{n} \sup_{x \in \R} \left| \hat{F}_j(x) - \pi_j(\theta \leq x) \right| \leq \frac{1}{\sqrt{2n}} \sum_{j=0}^{M-1} \sqrt{ \frac{m_j}{n} \log\left( \frac{6n}{\delta m_j} \right) \wedge \frac{2m_j^2}{n}}.
\]
\end{lemma}

The second sum on the right side of \eqref{eq:decomp-sum} represents the error in learning the tour length distribution. For $i>0$ and $j \geq 0$, we have already defined $p_i = \Prob(T=i)$ and $P_j = \Prob(T>j) = \sum_{i>j} p_i$. Now, define the empirical frequencies
\[
\hat{p}_i = \frac{\#\{1 \leq k \leq K : W_k = i\}}{K}, \qquad \hat{P}_j = \sum_{i>j} \hat{p}_i = \frac{m_j}{K}.
\]
Below, we bound the second sum on the right side of \eqref{eq:decomp-sum} in terms of the 1-Wasserstein distance between the distributions $\hat{p} = (\hat{p}_i)_{i \geq 1}$ and $p = (p_i)_{i \geq 1}$, which has the formula
\[
\W_1(\hat{p},p) = \sum_{j=0}^\infty |\hat{P}_j - P_j|.
\]

\begin{lemma} \label{lemma:Wasserstein}
\[
\sum_{j=0}^\infty \left| \frac{m_j}{n} - q_j \right| \leq \frac{2K}{n} \W_1(\hat{p},p).
\]
\end{lemma}

To control $\W_1(\hat{p},p)$, we first use a standard argument \cite[Theorem 3.2]{BL19} to show that:

\begin{lemma} \label{lemma:RMS}
\[
\sqrt{\eee(\W_1^2(\hat{p},p))} \leq \frac{1}{\sqrt{K}} \sum_{j=0}^\infty \sqrt{P_j(1-P_j)}.
\]
\end{lemma}

Since we observed $K$ independent tour lengths with a maximum length of $M$, we can say with high probability that tours of length greater than $M$ are rare.

\begin{lemma} \label{lemma:PM-bound}
With probability at least $1 - \delta/3$,
\begin{equation} \label{eq:PM-bound}
P_M \leq \frac{1}{K} \log\left( \frac{3}{\delta} \right).
\end{equation}
\end{lemma}

We split the sum on the right side of Lemma \ref{lemma:RMS} into three pieces. For $j<M$, we simply use that $P_j(1-P_j) \leq 1/4$. For $M \leq j < J$, where
\begin{equation} \label{eq:J-def}
J = \frac{1}{\gamma} \log(BK),
\end{equation}
we use $P_j \leq P_M$ and the bound in Lemma \ref{lemma:PM-bound}. Finally, for $j \geq J$, we use that
\[
P_j = \Prob(T>j) \leq Be^{-\gamma j}.
\]
Putting it all together leads to the following bound.

\begin{lemma} \label{lemma:combined-bound}
With probability at least $1 - \delta/3$,
\[
\sqrt{\eee(\W_1^2(\hat{p},p))} \leq \frac{M}{2\sqrt{K}} + \frac{1}{K} \left( \frac{\log(BK)}{\gamma} \sqrt{\log\left(\frac{3}{\delta} \right)} + \frac{1}{1 - e^{-\gamma/2}} \right).
\]
\end{lemma}

The next step is to apply a concentration inequality for $\W_1(\hat{p},p)$. Using an extension of McDiarmid's inequality due to Wintenberger \cite{W17} that allows for unbounded differences, we can show that:

\begin{lemma} \label{lemma:Wint-bound}
With probability at least $1 - \delta/3$,
\begin{equation} \label{eq:Wint}
\W_1(\hat{p},p) - \eee(\W_1(\hat{p},p)) \leq \frac{1}{\sqrt{K}} \left( \log\left( \frac{3}{\delta} \right) + \frac{\eee(T^2)}{2} + \frac{1}{2K} \sum_{k=1}^K W_k^2 \right).
\end{equation}
\end{lemma}
Unfortunately, the $\eee(T^2)$ term means that the right side of \eqref{eq:Wint} depends suboptimally on $B,\gamma$. Since we could not resolve this issue, we instead use Markov's inequality. This results in a bound with optimal dependence on $B,\gamma$ but not on $\delta$.

\begin{lemma} \label{lemma:Markov}
With probability at least $1 - \delta/3$,
\[
\W_1(\hat{p},p) \leq \frac{\sqrt{\eee(\W_1^2(\hat{p},p))}}{\sqrt{\delta/3}}.
\]
\end{lemma}

We combine Lemmas \ref{lemma:Markov} and \ref{lemma:combined-bound} using a union bound. Then, we plug the result into Lemma \ref{lemma:Wasserstein}. With probability at least $1 - 2\delta/3$,
\[
\sum_{j=0}^\infty \left| \frac{m_j}{n} - q_j \right| \leq \frac{1}{\sqrt{\delta/3}} \cdot \frac{M\sqrt{K}}{n} + \frac{1}{\sqrt{\delta/3}} \cdot \frac{2}{n} \left( \frac{\log(BK)}{\gamma} \sqrt{\log\left(\frac{3}{\delta} \right)} + \frac{1}{1 - e^{-\gamma/2}} \right).
\]
Finally, we combine with Lemma \ref{lemma:dkw-result} using another union bound and invoke Lemma \ref{lemma:decomp-sum}. This concludes the proof of Theorem \ref{thm:main}. \qed

\section{Proofs of lemmas} \label{sec:proofs-of-lemmas}

\begin{proof}[Proof of Lemma \ref{lemma:Jensen}]
Let
\[
f(x) = \sqrt{x \log\left( \frac{6}{\delta x} \right)}.
\]
A direct computation shows that this function is concave on $(0,1)$. We have
\[
\sum_{j=0}^{M-1} \sqrt{ \frac{m_j}{n} \log\left( \frac{6n}{\delta m_j} \right) \wedge \frac{2m_j^2}{n}} \leq \sum_{j=0}^{M-1} \sqrt{ \frac{m_j}{n} \log\left( \frac{6n}{\delta m_j} \right) } = \sum_{j=0}^{M-1} f\left( \frac{m_j}{n} \right).
\]
Since $f$ is concave on $(0,1)$, for $x_0,\ldots,x_{M-1} \in (0,1)$ we have by Jensen's inequality
\[
\frac{1}{M} \sum_{j=0}^{M-1} f(x_j) \leq f \left( \frac{x_0 + \cdots + x_{M-1}}{M} \right).
\]
Plugging in $x_j = m_j/n$, we observe that $x_0 + \cdots + x_{M-1} = 1$. Thus,
\[
\sum_{j=0}^{M-1} f\left( \frac{m_j}{n} \right) \leq M f\left(\frac{1}{M}\right) = \sqrt{ M \log\left(\frac{6M}{\delta} \right) }. \qedhere
\]
\end{proof}

\begin{proof}[Proof of Lemma \ref{lemma:decomp-sum}]
We compute using \eqref{eq:cdf-decomp} and \eqref{eq:emp-cdf-decomp} that
\[
\begin{split}
\hat{F}(x) - \pi(\theta \leq x) &= \sum_{j=0}^{M-1} \frac{m_j}{n} \hat{F}_j(x) - \sum_{j=0}^\infty q_j \pi_j(\theta \leq x) \\
&= \sum_{j=0}^{M-1} \frac{m_j}{n}[\hat{F}_j(x) - \pi_j(\theta \leq x)] + \sum_{j=0}^\infty \left[ \frac{m_j}{n} - q_j \right] \pi_j(\theta \leq x),
\end{split}
\]
recalling that $m_j = 0$ for $j \geq M$. It follows that
\[
\sup_{x \in \R} |\hat{F}(x) - \pi(\theta \leq x)| \leq \sum_{j=0}^{M-1} \frac{m_j}{n} \sup_{x \in \R} |\hat{F}_j(x) - \pi_j(\theta \leq x)| + \sum_{j=0}^\infty \left| \frac{m_j}{n} - q_j \right|. \qedhere
\]
\end{proof}

\begin{proof}[Proof of Lemma \ref{lemma:dkw-result}]
We use the classical DKW inequality for iid data \cite{DKW56,M90,R24} separately for each $j$. Given any constants $\e_j > 0$, the DKW inequality says that
\begin{equation} \label{eq:DKW-wedge}
\sup_{x \in \R} \left| \hat{F}_j(x) - \pi_j(\theta \leq x) \right| \leq \e_j \wedge 1
\end{equation}
with probability at least $1 - 2\exp(-2m_j \e_j^2)$. (On the right side of \eqref{eq:DKW-wedge}, the $\e_j$ is the DKW bound and the $1$ is trivial.)

We will set
\[
\e_j^2 = \frac{1}{2m_j} \log\left( \frac{6n}{\delta m_j} \right), \qquad \text{so that} \qquad 2\exp(-2m_j \e_j^2) = \frac{\delta m_j}{3n}.
\]
By a union bound, \eqref{eq:DKW-wedge} holds for all $0 \leq j \leq M-1$ with probability at least
\[
1 - \sum_{j=0}^{M-1} \frac{\delta m_j}{3n} = 1 - \frac{\delta}{3}.
\]
This gives us
\[
\begin{split}
\sum_{j=0}^{M-1} \frac{m_j}{n} \sup_{x \in \R} \left| \hat{F}_j(x) - \pi_j(\theta \leq x) \right| &\leq \sum_{j=0}^{M-1} \frac{m_j}{n} \cdot \left( \frac{1}{\sqrt{2m_j}} \sqrt{ \log\left( \frac{6n}{\delta m_j} \right)} \wedge 1 \right) \\
&= \frac{1}{\sqrt{2n}} \sum_{j=0}^{M-1} \sqrt{ \frac{m_j}{n} \log\left( \frac{6n}{\delta m_j} \right) \wedge \frac{2m_j^2}{n}}. \qedhere
\end{split}
\]
\end{proof}

\begin{proof}[Proof of Lemma \ref{lemma:Wasserstein}]
We know that
\[
q_j = \frac{P_j}{\eee(T)}, \qquad \qquad \frac{m_j}{n} = \frac{m_j/K}{n/K} = \frac{\hat{P}_j}{n/K}.
\]
Thus,
\[
\begin{split}
\sum_{j=0}^\infty \left| \frac{m_j}{n} - q_j \right| &= \sum_{j=0}^\infty \left| \frac{\hat{P}_j}{n/K} - \frac{P_j}{\eee(T)} \right| \\
&= \sum_{j=0}^\infty \left| \frac{1}{n/K} \left( \hat{P}_j - P_j \right) + P_j \left( \frac{1}{n/K} - \frac{1}{\eee(T)} \right) \right| \\
&\leq \frac{K}{n} \sum_{j=0}^\infty \left| \hat{P}_j - P_j \right| + \left| \frac{1}{n/K} - \frac{1}{\eee(T)} \right| \sum_{j=0}^\infty P_j.
\end{split}
\]
Since
\[
\sum_{j=0}^\infty P_j = \eee(T) \qquad \text{and} \qquad \sum_{j=0}^\infty \hat{P}_j = \frac{n}{K},
\]
we have
\[
\begin{split}
\left| \frac{1}{n/K} - \frac{1}{\eee(T)} \right| \sum_{j=0}^\infty P_j &= \left| \frac{1}{n/K} - \frac{1}{\eee(T)} \right| \eee(T) = \frac{1}{n/K} \left| \eee(T) - \frac{n}{K} \right| \\
&= \frac{K}{n} \left| \sum_{j=0}^\infty P_j - \sum_{j=0}^\infty \hat{P}_j \right| \leq \frac{K}{n} \sum_{j=0}^\infty \left| P_j - \hat{P}_j \right|.
\end{split}
\]
It follows that
\[
\sum_{j=0}^\infty \left| \frac{m_j}{n} - q_j \right| \leq \frac{2K}{n} \sum_{j=0}^\infty \left| \hat{P}_j - P_j \right| = \frac{2K}{n} \W_1(\hat{p},p). \qedhere
\]
\end{proof}

\begin{proof}[Proof of Lemma \ref{lemma:RMS}]
Following the argument in \cite[Theorem 3.2]{BL19}, we write
\[
\W_1(\hat{p},p) = \sum_{j=0}^\infty |\hat{P}_j - P_j|.
\]
By the triangle inequality in $L^2$,
\[
\sqrt{\eee(\W_1^2(\hat{p},p))} \leq \sum_{j=0}^\infty \sqrt{\eee(|\hat{P}_j - P_j|^2)}.
\]
Each $\hat{P}_j$ is a random variable whose expected value is $P_j$. Thus,
\[
\eee(|\hat{P}_j - P_j|^2) = \Var(\hat{P}_j).
\]
We know that $\hat{P}_j = m_j/K$, where $m_j \sim \textrm{Binomial}(K, P_j)$, since each tour length has an independent $P_j$ chance to be longer than $j$. Therefore,
\[
\Var(\hat{P}_j) = \frac{P_j(1-P_j)}{K}.
\]
We conclude that
\[
\sqrt{\eee(\W_1^2(\hat{p},p))} \leq \sum_{j=0}^\infty \sqrt{\Var(\hat{P}_j)} = \frac{1}{\sqrt{K}} \sum_{j=0}^\infty \sqrt{P_j(1-P_j)}. \qedhere
\]
\end{proof}

\begin{proof}[Proof of Lemma \ref{lemma:PM-bound}]
Imagine a selection process for $M$ based on auxiliary independent $\textrm{Uniform}(0,1)$ random variables $U_1,\ldots,U_K$. Namely, after sampling each $U_k$, the length $W_k$ of the $k$-th tour is chosen by taking the $U_k$-th quantile of the tour length distribution. For each $0<\beta<1$,
\[
\Prob(\max(U_1,\ldots,U_K) \leq 1-\beta) = (1-\beta)^K.
\]
We choose
\[
\beta = \frac{1}{K} \log\left( \frac{3}{\delta} \right)
\]
so that
\[
(1-\beta)^K \leq e^{-\beta K} = \frac{\delta}{3}.
\]
With probability at least $1 - \delta/3$, $\max(U_1,\ldots,U_K)$ is at least $1-\beta$. This means that an independently chosen sample from the tour length distribution has probability at most $\beta$ to exceed $M$. In other words, $P_M \leq \beta$.
\end{proof}

\begin{proof}[Proof of Lemma \ref{lemma:combined-bound}]
Define $J$ as in \eqref{eq:J-def}. With probability at least $1 - \delta/3$, the bound \eqref{eq:PM-bound} from Lemma \ref{lemma:PM-bound} holds, and we have
\[
P_j(1-P_j) \leq \begin{cases} 1/4 & \text{if } j < M, \\ \frac{1}{K} \log(3/\delta) & \text{if } M \leq j < J, \\ Be^{-\gamma j} & \text{if } j \geq J. \end{cases}
\]
It follows that
\[
\begin{split}
\sum_{j=0}^\infty \sqrt{P_j(1-P_j)} &\leq \frac{M}{2} + \sum_{M\leq j < J} \frac{1}{\sqrt{K}} \sqrt{ \log\left( \frac{3}{\delta} \right) } + \sum_{j \geq J} \sqrt{B} e^{-\gamma j/2} \\
&\leq \frac{M}{2} + \frac{J}{\sqrt{K}} \sqrt{\log\left( \frac{3}{\delta} \right)} + \sqrt{Be^{-\gamma J}} \sum_{j \geq J} e^{-\gamma(j-J)/2} \\
&\leq \frac{M}{2} + \frac{\log(BK)}{\gamma \sqrt{K}} \sqrt{\log\left( \frac{3}{\delta} \right)} + \frac{1}{\sqrt{K}} \cdot \frac{1}{1 - e^{-\gamma/2}},
\end{split}
\]
using in the last line that $Be^{-\gamma J} = 1/K$. Plug into Lemma \ref{lemma:RMS} to finish the proof.
\end{proof}

\begin{proof}[Proof of Lemma \ref{lemma:Wint-bound}]
The empirical tour length distribution $\hat{p}$ depends on the lengths $W_1,\ldots,W_K$. Given $w = (w_1,\ldots,w_K)$, write $\hat{p}(w)$ for the empirical tour length distribution given that $W_1 = w_1,\ldots, W_K = w_K$. Define
\[
f(w) = \W_1(\hat{p}(w), p).
\]
Now, let $w' = (w_1,\ldots,w_{k-1},w'_k,w_{k+1},\ldots,w_K)$ be the same as $w$ except that the single entry $w_k$ has been replaced by $w'_k$. We have
\[
|f(w) - f(w')| = |\W_1(\hat{p}(w),p) - \W_1(\hat{p}(w'),p)| \leq \W_1(\hat{p}(w),\hat{p}(w')) \leq \frac{1}{K}|w_k - w'_k|.
\]
This means that the function $f$ satisfies inequality (3) in \cite{W17} (inequality (2.3) in the arXiv version) with $L_k = 1/K$ and $V(x) = x$. Remark 2.1 in \cite{W17} now says that
\[
\eee\left[ \exp\left( \lambda(f - \eee(f)) - \frac{\lambda^2}{2} \sum_{k=1}^K \frac{1}{K^2} \left( \eee(W_k^2) + W_k^2 \right) \right) \right] \leq 1
\]
for all $\lambda > 0$. By Markov's inequality, with probability at least $1 - \delta/3$,
\[
\lambda(f - \eee(f)) - \frac{\lambda^2}{2} \sum_{k=1}^K \frac{1}{K^2} \left( \eee(W_k^2) + W_k^2 \right) \leq \log\left( \frac{3}{\delta} \right)
\]
and therefore
\begin{equation} \label{eq:Wint-consequence}
f - \eee(f) \leq \frac{1}{\lambda} \log\left( \frac{3}{\delta} \right) + \frac{\lambda}{2} \sum_{k=1}^K \frac{1}{K^2} \left( \eee(W_k^2) + W_k^2 \right).
\end{equation}

The left side of \eqref{eq:Wint-consequence} is $\W_1(\hat{p},p) - \eee(\W_1(\hat{p},p))$. For the right side, we take $\lambda = \sqrt{K}$ and observe that $\eee(W_k^2) = \eee(T^2)$ for all $k$. Thus, with probability at least $1 - \delta/3$,
\[
\W_1(\hat{p},p) - \eee(\W_1(\hat{p},p)) \leq \frac{1}{\sqrt{K}} \log\left( \frac{3}{\delta} \right) + \frac{1}{2\sqrt{K}} \left( \eee(T^2) + \frac{1}{K} \sum_{k=1}^K W_k^2 \right). \qedhere
\]
\end{proof}

\begin{proof}[Proof of Lemma \ref{lemma:Markov}]
By Markov's inequality,
\[
\Prob\left( \W_1(\hat{p},p) \geq \frac{\sqrt{\eee(\W_1^2(\hat{p},p))}}{\sqrt{\delta/3}} \right) = \Prob\left( \W_1^2(\hat{p},p) \geq \frac{ \eee(\W_1^2(\hat{p},p))}{\delta/3} \right) \leq \frac{\delta}{3}. \qedhere
\]
\end{proof}

\end{document}